\documentclass[12pt, reqno, a4paper]{amsart}
\usepackage{amsfonts, amsmath, amsthm, amstext, amssymb}
\usepackage{mathtools}
\usepackage{url}
\usepackage{paralist}
\usepackage{tikz}\usetikzlibrary{arrows}
\usepackage{tikz-cd}
\usepackage[english]{babel}
\usepackage{framed}
\usepackage{amsrefs}
\usepackage[marginratio=1:1,height=660pt,width=410pt,tmargin=100pt]{geometry}
\usepackage{lipsum}
\usepackage{layout}
\usepackage{hyperref}

\DeclareFontFamily{U}{mathx}{\hyphenchar\font45}
\DeclareFontShape{U}{mathx}{m}{n}{
      <5> <6> <7> <8> <9> <10>
      <10.95> <12> <14.4> <17.28> <20.74> <24.88>
      mathx10
      }{}
\DeclareSymbolFont{mathx}{U}{mathx}{m}{n}
\DeclareFontSubstitution{U}{mathx}{m}{n}
\DeclareMathAccent{\widecheck}{0}{mathx}{"71}

\usepackage[OT2,T1]{fontenc}
\DeclareSymbolFont{cyrletters}{OT2}{wncyr}{m}{n}
\DeclareMathSymbol{\Sha}{\mathalpha}{cyrletters}{"58}

\DeclareMathOperator{\Spec} {Spec}

\DeclareMathOperator{\poly} {poly}

\DeclareMathOperator{\rk} {rk}

\DeclareMathOperator{\CC} {\mathbb{C}}

\DeclareMathOperator{\NN} {\mathbb{N}}

\DeclareMathOperator{\QQ} {\mathbb{Q}}

\DeclareMathOperator{\ZZ} {\mathbb{Z}}

\newtheorem{lemma}{Lemma}[section]
\newtheorem{thm}[lemma]{Theorem}

\theoremstyle{definition}\newtheorem{definition}[lemma]{Definition}
\theoremstyle{definition}
\theoremstyle{definition}

\newcommand{\mb}[1]{\mathbb{#1}}
\newcommand{\mc}[1]{\mathcal{#1}}
\newcommand{\mf}[1]{\mathfrak{#1}}
\newcommand{\ra}{\rightarrow}
\newcommand{\hra}{\hookrightarrow}
\newcommand{\ol}[1]{\overline{#1}}

\newcommand{\tb}[1]{\textbf{#1}}

\begin{document}
\title[Hypersurface coverings of regular rational points]{Slope-determinant method, complex cellular structures and hypersurface coverings of regular rational points}
\author{Kenneth Chung Tak Chiu}
\address{Department of Mathematics and New Cornerstone Science Laboratory, The University of Hong Kong.}
\address{Max-Planck-Institut f\"{u}r Mathematik, Bonn, Germany}
\email{kennethct.chiu@alumni.utoronto.ca}
\begin{abstract}
We use the determinant method of Bombieri-Pila and Heath-Brown and its Arakelov reformulation by Chen utilizing Bost's slope method to estimate the number of hypersurfaces required to cover the regular rational points with bounded Arakelov height on a projective variety.
Using complex cellular structures introduced by Binyamini-Novikov, we replace the usual subpolynomial factor by a polylogarithmic factor in the estimation.
\end{abstract}

\maketitle

\section{Introduction}
\subsection{Statement of result}
Let $K$ be a finite extension of $\QQ$.
Let $M$ be a positive integer.
Let $\iota: X\hookrightarrow \mb{P}_K^M$ be a closed immersion, where $X$ is an irreducible subvariety of dimension $n$ and degree $d$.
Let $\mc{X}$ be the Zariski closure of $X$ in $\mb{P}^M_{\mc{O}_K}$.
Let $\ol{\mc{O}_{\mc{X}}(1)}$ be the pullback Hermitian line bundle on $\mc{X}$ of the standard Hermitian line bundle $\ol{\mc{L}}:=\ol{\mc{O}(1)}$ on $\mb{P}^M_{\mc{O}_{K}}$ equipped with the Fubini-Study metric. 
The Hermitian line bundle $\ol{\mc{L}}$ leads to a notion of exponential height in Arakelov geometry, see Definition \ref{height of a point}. We will simply refer to it as height. In this paper, we prove the following:
 
\begin{thm}\label{Arakelov Bombieri-Pila}
There exists a constant $B_0>0$ depending only on $M,n,d$ and $K$ such that for any $B>B_0$,
the set $S_1(X,B)$ of regular $K$-points of $X$ with height $\leq B$ is covered by $N$ hypersurfaces defined over $K$ of degree $\lceil d\log B\rceil$, none of which contain $X$,
where  
$$N\leq \poly_{M,n,K}(d,\log B) \cdot B^{(n+1)d^{-1/n}}.$$
\end{thm}

\subsection{Previous works and strategy of proof}
When $X$ is a curve and $M=2$, Heath-Brown \cite{Hea02} proved that the number of rational points on $X$ with bounded Weil height $B$ is at most $O_{\varepsilon, d}(B^{2/d+\varepsilon})$ for any $\varepsilon >0$.
This was later improved by Salberger \cite{Sal07} to $O_{M,d}(B^{2/d}\log B)$ for arbitrary $M$, and by Walsh \cite{Wal15} to $O_{M,d}(B^{2/d})$.
Binyamini, Cluckers and Kato proved in a recent paper \cite{BCK24} that the implicit coefficient in Salberger's bound can be made quadratic in $d$. 
For any dimension $n$, Broberg \cite{Bro04} estimated the number of hypersurfaces required to cover the rational points on $X$ with bounded Weil height.
Broberg's bound is of the form $O_{M,d,\varepsilon}(B^{(n+1)d^{-1/n}+\varepsilon})$ and the degrees of the hypersurfaces are $O_{M,d,\varepsilon}(1)$.
The proof strategy in Heath-Brown \cite{Hea02} consists of constructing an auxiliary hypersurface of low degree that covers rational points having the same reduction modulo a large prime number.
This idea, called the determinant method, is inspired by results of Bombieri-Pila \cite{BP89} and Pila \cite{Pil95}.
Bombieri-Pila \cite{BP89} proved conjectures raised by Sarnak \cite{Sar88}.

Chen \cite{Chen12b} estimated the number of hypersurfaces required to cover the regular rational points on $X$ with bounded Arakelov height. Chen's bound is of the form $CB^{(1+\varepsilon)(n+1)d^{-1/n}}$, where $C$ is a constant depending only on $M, n, d, K, \varepsilon$. Chen achieved this by reformulating the determinant method in the framework of Arakelov geometry using the evaluation map and the slope method by Bost \cite{Bos96}, \cite{Bos01}, \cite{Bos06}. 
Binyamini and Novikov \cite{BN19} introduced the notion of complex cellular structures and used it together with the usual determinant method to estimate the number of hypersurfaces required to cover the rational points on $X$ with bounded Weil height \cite[Theorem 6]{BN19}. Their bound is the form 
$\poly_{M}(d, \log B)\cdot B^{(n+1)d^{-1/n}}$ and the degree of the hypersurfaces is $\poly_M(d)\cdot \lceil \log B \rceil$.

Compared with the work of Binyamini-Novikov \cite{BN19}, we replace Weil height by Arakelov height and we work with arbitrary number fields.
However, we only consider regular rational points.
Compared with the work of Chen \cite{Chen12b}, the subpolynomial factor $CB^{\varepsilon (n+1) d^{-1/n}}$ is replaced by a polylogarithmic factor $\poly_{M,n,K}(d,\log B)$, although our $B_0$ is not explicit, and the degree of our hypersurfaces is $\lceil d \log B\rceil$ instead of $O_{M,n,d,\varepsilon}(1)$.

By comparing the height used in this paper with the height used in Binyamini-Novikov \cite{BN19}, it might be possible to deduce an Arakelov version similar to Theorem \ref{Arakelov Bombieri-Pila} from a generalization of Theorem 6 in \emph{op. cit.} to arbitrary number field. 
However, the degree of the hypersurfaces in such Arakelov version will depend on $[K:\QQ]$.

Our proof combines Chen's strategy \cite{Chen12b} with the strategy of Binyamini-Novikov \cite{BN19}.
Lemma 93 of Binyamini-Novikov \cite{BN19} gives us sharper bounds for the Archimedean norms of the determinant of the evaluation map when the rational points are in the same complex cell. 
By taking the degree of the hypersurfaces to be $d\log B$, certain terms and factors in these sharper bounds  offset against those coming from the finite places in the Arakelov setting, so we can choose $\delta$ in a way that the size of the covering by complex cells admitting $\delta$-extension is $\poly_{M,n,K}(d,\log B)$, without the factor $B^{(n+1)d^{-1/n}}$ which appears in \cite[B.2.2]{BN19}. This avoids the issue of overcounting as the factor $B^{(n+1)d^{-1/n}}$ appears when we count points over finite fields using Chen's strategy.
We also replace the use of Bertrand's postulate in \cite[Theorem 4.2]{Chen12b} by the use of a theorem on primes between cubes by Dudek \cite{Dud16}, 
in order to obtain polynomial dependence on $d$ in the polylogarithmic factor of the bound.

\subsection{Further problems}
Let us mention some further problems for future investigations:

\begin{itemize}
\item One may try to use similar strategy to bound the number of rational points on a planar curve of degree $d$ with height at most $d$. One expects the number to be $\ll_K d^2$. This would refine \cite[Theorem B]{Chen12b}.
\item It would be interesting to obtain a version of Theorem \ref{Arakelov Bombieri-Pila} for all $K$-points. 
\item As mentioned earlier, Walsh \cite{Wal15} eliminated the logarithmic factor in the curve case. It would be interesting to study whether the polylogarithmic factor in Theorem \ref{Arakelov Bombieri-Pila} could also be eliminated. 
\item It is worth to consider other Hermitian line bundles and prove analogues of Theorem \ref{Arakelov Bombieri-Pila} for them.
\end{itemize}

\subsection{Notations in Arakelov geometry}
We first recall some notions in Arakelov geometry that were used in Chen's works \cite{Chen12a} and \cite{Chen12b}.
Let $K$ be a finite extension of $\QQ$.
An \emph{arithmetic projective variety} is an integral projective $\mc{O}_K$-scheme which is flat over $\Spec \mc{O}_K$.
A \emph{Hermitian vector bundle} on an arithmetic projective variety $X$ is a pair $\ol{E}=(E, \|\cdot\|_{\sigma})_{\sigma:K\hra \CC}$, where $E$ is a locally free $\mc{O}_X$-module of finite rank,
and for any embedding $K\hra \CC$, $\|\cdot\|_{\sigma}$ is a continuous Hermitian metric on $E_{\sigma}(\CC)$, invariant under the action of the complex conjugation.
The \emph{Arakelov degree} of a Hermitian vector bundle $\ol{E}$ on $\Spec \mc{O}_K$ is defined as
$$\widehat{\deg}(\ol{E}):=\log \# (E/(\mc{O}_Ks_1+\cdots +\mc{O}_K s_r))-\frac{1}{2}\sum_{\sigma: K\hra \CC}\log (\langle s_i,s_j\rangle_{\sigma}),$$
where $(s_1,\dots, s_r)\in E^r$ forms a basis of $E_K$ over $K$, where $r$ is the rank of $E$.

\begin{definition}
Let $\ol{E}$ be a non-zero Hermitian vector bundle on $\Spec \mc{O}_K$. The \emph{slope} of $\ol{E}$ is 
$$\widehat{\mu} (\ol{E}):=\frac{1}{[K:\QQ]} \frac{\widehat{\deg}(\ol{E})}{\rk(E)}.$$
Denote by $\widehat{\mu}_{\max}(\ol{E})$ the maximal values of slopes of all non-zero Hermitian subbundles of $\ol{E}$.
\end{definition}

\begin{definition}\label{height of a point}
Any point $P\in \mb{P}^M_K(K)$ gives rise to a unique $\mc{O}_K$-point $\mc{P}$ of $\mb{P}^M_{\mc{O}_K}$. Denote by $h(P)$ the slope of $\mc{P}^*\ol{\mc{L}}$.
The \emph{height} of $P$ is defined to be $H(P):=\exp ([K:\QQ] h(P))$.
Denote by $S(X,B)$ the subset of $X(K)$ consisting of points $P$ such that $H(P)\leq B$.
Denote by $S_1(X,B)$ the subset of $S(X,B)$ of regular points. 
\end{definition}

\begin{definition}
Let $X$ be an integral closed subscheme of $\mb{P}^M_K$ of dimension $n$ and $d$. 
The \emph{Arakelov height} of $X$ is defined to be 
$$h_{\ol{\mc{L}}}(X):=\frac{1}{[K:\QQ]}\widehat{\deg}(\widehat{c_1} (\ol{\mc{L}})^{d+1} \cdot [\mc{X}]),$$
where $\mc{X}$ is the Zariski closure of $X$ in $\mb{P}^M_{\mc{O}_K}$.
\end{definition}

For any maximal ideal $\mf{p}$ of $\mc{O}_K$, denote by $\mb{F}_{\mf{p}}$ its residue field, and by $N_{\mf{p}}$ the cardinality of $\mb{F}_{\mf{p}}$.  
For any point $\xi\in \mc{X}(\mb{F}_{\mf{p}})$, 
let $\mf{m}_{\xi}$ be the maximal ideal of the local ring $\mc{O}_{\mc{X},\xi}$, 
the \emph{Hilbert-Samuel function} $H_{\xi}:\NN\ra\NN$ of $\mc{O}_{\mc{X},\xi}/\mf{m}_{\xi}$ is defined by 
$$H_{\xi}(k)=\rk_{\mb{F}_p} ((\mf{m}_{\xi}/\mf{p}\mc{O}_{\mc{X},\xi})^k/(\mf{m}_{\xi}/\mf{p}\mc{O}_{\mc{X},\xi})^{k+1}).$$
Let $(q_{\xi}(m))_{m\geq 1}$ be the increasing sequence of integers such that the integer $k\in \NN$ appears exactly $H_{\xi}(k)$ times.
Let $Q_{\xi}(m)=q_{\xi}(1)+\cdots +q_{\xi}(m)$.

Since $\mc{X}$ is a proper $\mc{O}_K$-scheme, any point in $X(K)$ extend to an $\mc{O}_K$-point.
For any maximal ideal $\mf{p}$ of $\mc{O}_K$ and any $a\in \ZZ_{>0}$, denote by $A^{(a)}_{\mf{p}}$ the Artinian local ring $\mc{O}_{K,\mf{p}}/\mf{p}^a\mc{O}_{K,\mf{p}}$.
For any $\eta\in \mc{X}(A^{(a)}_{\mf{p}})$, denote by $S(X,B,\eta)$ the set of points in $S(X,B)$ whose reduction modulo $\mf{p}^a$ coincides with $\eta$.
For any maximal ideal $\mf{p}$ of $\mc{O}_K$, define 
$$S_1(X,B, \mf{p}):= \bigcup_{\text{regular }\xi\in \mc{X}(\mb{F}_{\mf{p}})} S(X, B, \xi).$$

\begin{definition}
Let $\ol{E}$ be a Hermitian vector bundle on $\Spec \mc{O}_K$. 
The \emph{$r$-th wedge product} $\Lambda^ r\ol{E}$ is defined to be the Hermitian vector bundle 
$$(\Lambda^r E, (\|\cdot \|_{\Lambda^r \ol{E}, \sigma })_{\sigma:K\hra \CC}),$$
where 
$$\|\eta \|_{\Lambda^r \ol{E}, \sigma }:=\inf \{\|x_1\|_{\ol{E}, \sigma}\cdots \|x_r\|_{\ol{E}, \sigma} : x_1,\dots, x_r \in E\otimes_K\CC_{\sigma}, \eta=x_1\wedge\cdots \wedge x_r\}.$$
\end{definition}

\subsection{Acknowledgement}
I thank Gal Binyamini for discussions that motivate this work.
I thank Zhenlin Ran, Jiacheng Xia and Roy Zhao for discussions on Arakelov geometry.
I also thank Tuen Wai Ng and Kwok Kin Wong for discussions on geometric function theory.
This work began while the author was staying at the Max-Planck-Institut f\"{u}r Mathematik in Bonn. 
It was partially supported by the New Cornerstone Science Foundation through the New Cornerstone Investigator Program awarded to Professor Xuhua He.

\section{Complex cellular covers}
\subsection{Complex cellular structures}
We recall the notion of complex cellular structures introduced by Binyamini-Novikov \cite[\S 2]{BN19}.
For $r\in \CC$ (resp. $r_1,r_2\in \CC$) with $|r|>0$ (resp. $|r_2|>|r_1|>0$), set 
$$D(r):=\{z\in \CC: |z|<|r|\}, \quad D_{\circ} (r):=\{z\in \CC:0<|z|<|r|\},$$
$$A(r_1,r_2):=\{z\in \CC: |r_1|<|z|<|r_2|\}, \quad *:=\{0\}.$$
For any $0<\delta<1$, define the $\delta$-extensions by
$$D^{\delta}:=D(\delta^{-1}r),\quad D^{\delta}_{\circ}:= D_{\circ}(\delta^{-1}r),$$
$$A^{\delta}(r_1,r_2):=A(\delta r_1, \delta^{-1}r_2), \quad *^{\delta}:=*.$$
Let $\mc{X}$ and $\mc{Y}$ be sets. 
Let $\mc{F}:\mc{X}\ra 2^{\mc{Y}}$ be a map taking points of $\mc{X}$ to subsets of $\mc{Y}$.
We denote
$$\mc{X}\odot \mc{F}:=\{(x,y): x\in \mc{X}, y\in \mc{F}(x)\}.$$
If $U$ is a complex manifold, we denote by $\mc{O}_b(U)$ the set of bounded holomorphic functions on $U$.
We denote $\tb{z}_{1..\ell}:=(z_1,\dots, z_\ell)$.
\begin{definition}
A \emph{complex cell $\mc{C}$  of length $0$} is the point $\CC^0$. 
The \emph{type} of $\mc{C}$ is the empty word. 
A \emph{complex cell $\mc{C}$ of length $\ell+1$} has the form $\mc{C}_{1..\ell}\odot \mc{F}$, 
where the base $\mc{C}_{1..\ell}$ is a cell of length $\ell$, and the fiber $\mc{F}$ is one of $*$, $D(r)$, $D_{\circ}(r)$, $A(r_1,r_2)$, where $r\in \mc{O}_b(\mc{C}_{1..\ell})$ satisfies $|r(\tb{z}_{1..\ell})|>0$ for $\tb{z}_{1..\ell}\in \mc{C}_{1..\ell}$; and $r_1,r_2\in \mc{O}_b(\mc{C}_{1..\ell})$ satisfy $0<|r_1(\tb{z}_{1..\ell})|<|r_2(\tb{z}_{1..\ell})|$ for $\tb{z}_{1..\ell}\in \mc{C}_{1..\ell}$.
The \emph{type} $\mc{T}(\mc{C})$ is $\mc{T}(\mc{C}_{1..\ell})$ followed by the type of the fiber $\mc{F}$.
\end{definition}

\begin{definition}
Let $\pmb{\delta}\in (0,1)^\ell$. The notion of \emph{$\pmb{\delta}$-extension} of a cell of length $\ell$ is defined inductively as follows:
A complex cell of length $0$ is defined to be its own {$\pmb{\delta}$-extension}.
A cell $\mc{C}$ of length $\ell+1$ admits a $\pmb{\delta}$-extension $\mc{C}^{\pmb{\delta}}:= \mc{C}_{1..\delta}^{\pmb{\delta}_{1..\ell}}\odot \mc{F}^{\pmb{\delta}_{\ell+1}}$ if $\mc{C}_{1..\ell}$ admits a $\pmb{\delta}_{1..\ell}$-extension, and if the function $r$ (resp. $r_1,r_2$) involved in $\mc{F}$ admits holomorphic continuation to $\mc{C}_{1..\ell}^{\pmb{\delta}_{1..\ell}}$ and satisfies $|r(\tb{z}_{1..\ell})|>0$ (resp. $0<|r_1(\tb{z}_{1..\ell})|<|r_2(\tb{z}_{1..\ell})|$) in this larger domain.
\end{definition}

\begin{definition}
Let $\mc{C}$ and $\widehat{\mc{C}}$ be two cells of length $\ell$. 
A holomorphic map $f: \mc{C}\ra \widehat{\mc{C}}$ is \emph{cellular} if it takes the form 
$w_j=\phi_j(\tb{z}_{1..j})$, 
where $\phi_j\in \mc{O}_b(\mc{C}_{1..j})$ for $j=1,\dots, \ell$, and moreover $\phi_j$ is a monic polynomial of positive degree in $z_j$.
\end{definition}

\subsection{Complex cellular covers}\label{Complex cellular covers}
\begin{definition}[{\cite[Definition 12]{BN19}}]
Let $\mc{C}^{\pmb{\delta}}$ be a cell and $\{f_j: \mc{C}_j^{\pmb{\delta'}}\ra \mc{C}^{\pmb{\delta}}\}$ be a finite collection of cellular maps. 
This collection is said to be a \emph{$(\pmb{\delta}',\pmb{\delta})$-cellular cover} of $\mc{C}$ if 
$\mc{C}\subset \bigcup_j f_j(\mc{C}_j)$.
If $(\pmb{\delta}',\pmb{\delta})$ are clear from the context we will speak simply of cellular covers.
\end{definition}

Let $\tau_r:\mb{A}^{M}_K\hookrightarrow \mb{P}^{M}_K$ be the standard affine charts, where $r=0,\dots, M$.
Let $\sigma: K\hra \CC$ be an embedding.
Let $\Delta^{M}$ be the $M$-dimensional unit polydisk centered at the origin of $\mb{A}^{M}_{\CC}$.
Let $X_r:=\iota^{-1}(\tau_r(\Delta^M))$.
By dividing the maximum of the norms of the projective coordinates, we know that $X(\CC)=\bigcup_{r=0}^MX_r$.

As in the notations of \cite[\S2.1]{BN19}, let $(\Delta^{M})^{1/2}$ be the product of $M$ open disks of radius $2$.
By the cellular parametrization theorem of Binyamini-Novikov \cite[Theorem 8]{BN19}, we obtain a cellular cover 
$$\{\mc{C}^{1/2}_{\gamma}\ra  \tau_r^{-1}(\iota(X(\CC)))\cap (\Delta^{M})^{1/2}\}$$ 
of size $\poly_{M}(d)$ admitting $1/2$-extensions, where $d$ is the degree of $X$.
Let $\mu$ be a positive integer.
Fix $0<\delta<1/2$.  
By \cite[Lemma 94]{BN19},  for each cell $\mc{C}^{1/2}_{\gamma}$ in this cover,
there exists a cellular cover $\{\mc{C}_{\beta}^{\pmb{\delta}_\beta}\ra \mc{C}^{1/2}_{\gamma}\}$ of size 
$\poly_{M} (\mu \log (1/\delta))\cdot \delta^{-2n}$,
where $\pmb{\delta}_{\beta}$ is given by $\delta$ in the $D$-coordinates and $\delta^{E_m \mu^{1+1/m}}$ in the $A$, $D_{\circ}$ coordinates, where $m$ is the number of $D$-fibers of $\mc{C}_{\beta}$, and 
$$E_m:=\frac{m}{m+1}(m!)^{1/m}.$$
We denote by $\{\phi_\alpha\}_{\alpha\in I_{\sigma}}$ the set of all the compositions 
$$\mc{C}_{\beta}^{\pmb{\delta}_\beta}\ra \mc{C}^{1/2}_{\gamma}
\ra \tau_r^{-1}(\iota(X(\CC)))\cap (\Delta^{M})^{1/2}
\xrightarrow{\iota^{-1}\circ \tau_r} X(\CC)$$ of cellular maps obtained above.
Here $I_{\sigma}$ is the index set consisting of triples $(\beta, \gamma, r)$.
The cardinality of $I_{\sigma}$ is 
\begin{equation}\label{size of cellular cover}
(M+1)\cdot \poly_M(d)\cdot \poly_{M} (\mu \log (1/\delta))\cdot \delta^{-2n}.
\end{equation}
The set $X(\CC)$ is covered by the images of $\phi_{\alpha}$, $\alpha\in I_{\sigma}$.

\section{The evaluation map and slope method}
Let $D$ be a positive integer. 
We equip the $\mc{O}_K$-module $H^0(\mb{P}^M_{\mc{O}_K},\mc{O}(1)^{\otimes D})$ with the John norm attached to the sup-norm $\|\cdot\|_{\sigma, \sup}$ \cite[Notation 12]{Chen12b}.
Let $F_D$ be the saturation in $H^0(\mc{X}, \mc{O}_{\mc{X}}(1)^{\otimes D})$ of the image of $H^0(\mb{P}^M_K,\mc{O}(1)^{\otimes D})$ in $H^0(X, \mc{O}_X(1)^{\otimes D})$, i.e. the largest sub-$\mc{O}_K$-module $F_D$ of $H^0(\mc{X}, \mc{O}_{\mc{X}}(1)^{\otimes D})$ such that $F_{D,K}$ is the image of $H^0(\mb{P}^M_K,\mc{O}(1)^{\otimes D})$ in $H^0(X, \mc{O}_X(1)^{\otimes D})$.
Equip $F_D$ with the quotient metrics induced by the metrics of $H^0(\mb{P}^M_K,\mc{O}(1)^{\otimes D})$, so that $\ol{F_D}$ becomes a Hermitian vector bundle on $\Spec \mc{O}_K$.
Let $\sigma: K\hra \CC$ be as in Section \ref{Complex cellular covers}.
Let $\alpha\in I_{\sigma}$.
Let $\mu$ be a positive integer.
Let $P_1,\dots, P_\mu$ be $K$-points in the image of $\phi_{\alpha}$. 
Since $\mc{X}$ is a proper $\mc{O}_K$-scheme, the points $P_i$ extend to $\mc{O}_K$-points.
We have the \emph{evaluation map}
$$H^0(\mc{X},O_\mc{X}(1)^{\otimes D})  \ra  \bigoplus_{i= 1}^\mu P_i^* \mc{O}_\mc{X}(1)^{\otimes D}$$
defined by pulling back sections.
Let $$\Phi:F_{D,K}\ra \bigoplus_{i= 1}^\mu P_i^*\mc{O}_X(1)^{\otimes D}$$ be the restriction to the generic fiber $F_{D,K}$ of the evaluation map.
Let $r_1(D)$ be the rank of $F_D$.

\begin{lemma}[{\cite[\S 2.1]{Chen 12b}}]\label{slope, heights and determinants}
If $\Phi$ is an isomorphism, then
$$\frac{\widehat{\mu}(\ol{F_D})}{D}\leq \sup_{i\in I} \frac{\log H(P_i)}{[K:\QQ]}+\frac{1}{Dr_1(D)} h(\Lambda^{r_1(D)} \Phi),$$
where
$$h(\Lambda^{r_1(D)}\Phi):=\frac{1}{[K:\QQ]} \left(\sum_{\mf{p}} \log \| \Lambda^{r_1(D)} \Phi\|_{\mf{p}}+\sum_{\sigma:K\hra \CC} \log \|\Lambda^{r_1(D)} \Phi\|_{\sigma}\right).$$
\end{lemma}

Let $\mu$ be a positive integer, fix $0< \delta<1/2$, and let $E_m$ be as in Section \ref{Complex cellular covers}.
The following lemma rephrases Lemma 93 in Binyamini-Novikov \cite{BN19} in terms of the evaluation map.

\begin{lemma}\label{Upper bound of complex norm of determinant}
Suppose $P_1,\dots, P_\mu$ are $K$-points in the image of $\phi_{\alpha}$.
Suppose $\Phi$ is an isomorphism.
We have 
$$\|\Lambda^{\mu} \Phi\|_{\sigma}\leq \mu^{c_1(M)\mu} \delta^{E_m \mu^{1+1/m} -c_2(M)\mu},$$
where $m$ is the number of $D$-fibers of $\mc{C}_{\beta}$ (which is at most $n$) if $m\geq 1$; and $m=1$ if there are no $D$-fibers, and $c_1(M)>0$ and $c_2(M)$ are some constants depending only on $M$.
\end{lemma}

\begin{proof}
Suppose $s\in \bigwedge^\mu F_{D,K}$ with $\|s\|_{\sigma}=1$. Write $s:=s_1\wedge \cdots \wedge s_\mu$, where  $\|s_1\|_{\sigma}\cdots \|s_\mu\|_{\sigma}= 1+\varepsilon$ for some $\varepsilon >0$ for all $i$. 
By rescaling, we can assume that $\|s_i\|_{\sigma}=(1+\varepsilon)^{1/\mu}$.
Write $\alpha=(\beta, \gamma,r )$.
Let $x_1,\dots, x_{\mu}$ be the points in $\mc{C}_{\beta}^{\pmb{\delta}_\beta}$ such that $\phi_{\alpha}(x_i)=P_i$.
By \cite[Lemma 93]{BN19} and the fact that the sup-norm is at most the John norm \cite[Notation 12]{Chen12b}, the determinant of the matrix $(s_j(P_i))_{j,i}=(s_j(\phi_{\alpha}(x_i)))_{j,i}$ has complex norm bounded by $(1+\varepsilon)\mu^{O_M(\mu)} \delta^{E_m \mu^{1+1/m} -O_M(\mu)}$.
Thus, we have the same bound for $\|\Phi(s_1)\|_{\sigma}\cdots \|\Phi(s_\mu)\|_{\sigma}$.
Since $\varepsilon$ is arbitrary, we have $\|\Lambda^{\mu} \Phi(s)\|_{\sigma}\leq \mu^{O_M(\mu)} \delta^{E_m \mu^{1+1/m} -O_M(\mu)}$.
From the proof of Lemma 93 in Binyamini-Novikov \cite{BN19}, we see that $c_1(M)>0$.
\end{proof}

The following lemma is similar to \cite[Theorem 3.1]{Chen12b}, but we require the $K$-points to be contained in the same complex cell for each $\sigma:K\hra \CC$. The contributions from the Archimedean places in the estimate are given by Lemma \ref{Upper bound of complex norm of determinant} instead. 

\begin{lemma}\label{Slope, height, reduction modulo prime power}
Let $S=(\mf{p}_j)_{j\in J}$ be a finite family of maximal ideals of $K$ and $(a_j)_{j\in J}\in \ZZ_{>0}^{J}$.
For each $\mf{p}_j$, let $\eta_j$ be a point in $\mc{X}(A^{(a_j)}_{\mf{p}_j})$ whose reduction modulo $\mf{p}_j$ is denoted by $\xi_j$.
For each $\sigma:K\hra \CC$, choose $\alpha(\sigma)\in I_{\sigma}$.
Set $\mu=r_1(D)$. Let $0<\delta<1/2$.
Consider a family $(P_i)_{i\in I}$ of $K$-points such that 
\begin{itemize}
\item for any $i\in I$, and any $\sigma$, the point $P_i$ is in the image of $\phi_{\alpha(\sigma)}$, and
\item  for any $i\in I$, and any $j\in J$, the reduction of $P_i$ modulo $\mf{p}_j^{a_j}$ coincides with $\eta_j$.
\end{itemize}
Suppose there does not exist a hypersurface in $\mb{P}^M_K$ of degree $D$, not contained in $X$, containing all the points $P_i$, $i\in I$.
Then
\begin{align*}
\frac{\widehat{\mu}(\ol{F_D})}{D}\leq \sup_{i\in I} \frac{\log H(P_i)}{[K:\QQ]}
&+\frac{1}{Dr_1(D)}\log r_1(D)^{c_1(M)r_1(D)} \delta^{E_m r_1(D)^{1+1/m} -c_2(M)r_1(D)}
\\&-\frac{1}{[K:\QQ]}\sum_{j\in J}\frac{Q_{\xi_j}(r_1(D))}{Dr_1(D)}\log N_{\mf{p}_j}^{a_j}.
\end{align*}
\end{lemma}

\begin{proof}
Suppose there does not exist a hypersurface in $\mb{P}^M_K$ of degree $D$, not contained in $X$, containing all the points $P_i$, $i\in I$.
Then the evaluation map 
$$F_{D,K}\ra \bigoplus_{i\in I} P_i^*\mc{O}_X(1)^{\otimes D}$$
is injective.
By replacing $I$ by some finite subset $\{P_1,\dots, P_{\mu}\}$, where $\mu=r_1(D)$, we get an isomorphism from the evaluation map.
Since $P_i$ extends to $\mc{O}_K$-points, we have $\log \|\Lambda^{\mu} f\|_{\mf{p}}\leq 0$ for all finite places $\mf{p}$.
In the proof of Theorem 3.1 of \cite{Chen12b}, for all $j\in J$,
$$\log \|\Lambda^{\mu} f\|_{\mf{p}_j}\leq -Q_{\xi_j}(r_1(D))\log N_{\mf{p}_j}^{a_j}.$$
The lemma then follows from Lemma \ref{slope, heights and determinants} and Lemma \ref{Upper bound of complex norm of determinant}.
\end{proof}

\section{Hypersurface coverings of rational points}
The statement of the following lemma is analogous to \cite[Corollary 3.7]{Chen12b}. We remove the factor $1+\varepsilon$ in \emph{loc. cit.} by restricting to intersections of complex cells.
In the proof, we will apply the estimate in Lemma \ref{Slope, height, reduction modulo prime power}.
Certain terms and factors in this estimate will offset against those coming from the finite places.

\begin{lemma}\label{auxiliary hypersurface}
Let $(\mf{p}_j)_{j\in J}$ be a finite family of maximal ideals of $K$ and $(a_j)_{j\in J}\in \ZZ_{>0}^{J}$.
For any $j\in J$, let $\xi_j\in \mc{X}(\mb{F}_{\mf{p}_j})$ be a regular $K$-point of $\mc{X}(\mb{F}_{\mf{p}_j})$,
and let $\eta_j\in \mc{X}(A_{\mf{p}_j}^{(a_j)})$ whose reduction modulo $\mf{p}_j$ is $\xi_j$. Let 
\begin{equation}\label{Choosing delta}
\delta=\min\left\{\frac{1}{4}, \exp \left(\frac{-4(d-2)-2(n+3)d^{-1/n}-2d[K:\QQ] (2+\log(n+1))}{d[K:\QQ] d^{1/n} n!^{-1/n}}\right)\right\}.
\end{equation}
Let $D>d-2$ be a positive integer such that for all $1\leq m\leq n$,
\begin{equation}\label{positivity}
E_m \frac{d^{1/n}}{n!^{1/n}}\frac{D-d+2}{D}-\frac{c_2(M)}{D}>\frac{1}{2}E_m \frac{d^{1/n}}{n!^{1/n}}>0,
\end{equation}
\begin{equation}\label{assumption on D, bounded}
\frac{1}{D} c_1(M) \log \frac{d(D+n)^n}{n!}<1,
\end{equation}
and
\begin{equation}\label{another assumption on D}
 d^{1/n} \frac{n}{n+1}\left(1-\frac{d-2}{D}\right)-\frac{n+3}{2n+2}\cdot \frac{n}{D}>0.
\end{equation}
Set $B=e^{D/d}$ and $\mu=r_1(D)$.
For each $\sigma:K\hra \CC$, let $\phi_{\alpha(\sigma)}: \mc{C}_{\beta(\sigma)}^{\pmb{\delta}_{\beta(\sigma)}}\ra X(\CC)$ be one of the map constructed in Section \ref{Complex cellular covers}. 
If 
\begin{equation}\label{local ring mod prime cardinality assumption}
\sum_{j\in J} \log N^{a_j}_{\mf{p}_j} \cdot \frac{n}{n+1} d^{1/n}\geq
 \log B,
\end{equation}
then there exists a hypersurface of degree $D$ of $\mb{P}^M_K$ not containing $X$ which contains 
$$\left(\bigcap_{j\in J}S(X,B, \eta_j)\right)\cap \left(\bigcap_{\sigma:K\hra\CC} \phi_{\alpha(\sigma)}( \mc{C}_{\beta(\sigma)}^{\pmb{\delta}_{\beta(\sigma)}})\right).$$
\end{lemma}

\begin{proof}
Assume that such hypersurface does not exist. 
By Lemma \ref{Slope, height, reduction modulo prime power}, for some $1\leq m\leq n$, we have 
\begin{align*}
\frac{\widehat{\mu}(\ol{F_D})}{D}\leq \frac{\log B}{[K:\QQ]}
&+\frac{1}{Dr_1(D)}\log r_1(D)^{c_1(M)r_1(D)} \delta^{E_m r_1(D)^{1+1/m} -c_2(M) r_1(D)}
\\&-\frac{1}{[K:\QQ]}\sum_{j\in J}\frac{Q_{\xi_j}(r_1(D))}{Dr_1(D)}\log N_{\mf{p}_j}^{a_j}.
\end{align*}
Since $\xi_j$ is regular, by \cite[Prop. 3.5]{Chen12b}, 
$$\frac{Q_{\xi_j}(r_1(D))}{Dr_1(D)}\geq (n!)^{1/n} \frac{n}{n+1} \frac{r_1(D)^{1/n}}{D}-\frac{n+3}{2n+2}\cdot \frac{n}{D}.$$
In the proof of \cite[Cor. 3.7]{Chen12b}, which uses a result of Sombra \cite{Som97}, it is known that for $D>d-2$, 
$$r_1(D)\geq \frac{d(D-d+2)^n}{n!}.$$
Chardin \cite{Cha89} gave the following bound:
\begin{equation}\label{Chardin's bound}
r_1(D)\leq d\binom{D+n}{n}\leq \frac{d(D+n)^n}{n!}.
\end{equation}
Chen \cite[\S 2.4]{Chen12b} proved that
$$\frac{\widehat{\mu}(\ol{F_D})}{D}\geq -\frac{1}{2}\log(n+1).$$
Combining these bounds and using (\ref{assumption on D, bounded}), we have
\begin{align*}
0 &< \frac{1}{2}\log (n+1)+\frac{\log B}{[K:\QQ]}+1+(\log \delta)\left(\frac{E_m r_1(D)^{1/m}}{D}-\frac{c_2(M)}{D}\right)
\\&\quad -\left(d^{1/n} \frac{n}{n+1}\frac{D-d+2}{D}-\frac{n+3}{2n+2}\cdot \frac{n}{D}\right)\sum_{j\in J} \frac{\log N^{a_j}_{\mf{p_j}}}{[K:\QQ]}.
\end{align*}
Since $\log \delta<0$ and 
$$r_1(D)^{1/m}\geq r_1(D)^{1/n}\geq \frac{d^{1/n}}{n!^{1/n}}(D-d+2),$$
we have 
\begin{align*}
&\quad \left(d^{1/n} \frac{n}{n+1}\frac{D-d+2}{D}-\frac{n+3}{2n+2}\cdot \frac{n}{D}\right)\sum_{j\in J} \frac{\log N^{a_j}_{\mf{p_j}}}{[K:\QQ]}
\\&< \frac{1}{2}\log (n+1)+\frac{\log B}{[K:\QQ]}+1+(\log \delta)\left(E_m \frac{d^{1/n}}{n!^{1/n}}\frac{D-d+2}{D} -\frac{c_2(M)}{D}\right).
\end{align*}
Using (\ref{another assumption on D}), (\ref{local ring mod prime cardinality assumption}), and $D=d\log B$, we have
 \begin{align*}
&\quad d^{1/n} \frac{n}{n+1}\left(1-\frac{d-2}{D}\right)-\frac{n+3}{2n+2}\cdot \frac{n}{D}
\\&<d^{1/n}\frac{n}{n+1}\cdot \bigg(\frac{d[K:\QQ](2+\log (n+1))}{2D}+1
\\& \quad \quad\quad\quad \quad\quad\quad+(\log \delta)\frac{d[K:\QQ]}{D}\left(E_m \frac{d^{1/n}}{n!^{1/n}}\frac{D-d+2}{D} -\frac{c_2(M)}{D}\right)\bigg).
\end{align*}
Cancelling the term $d^{1/n}n/(n+1)$, using (\ref{positivity}) and the fact that $E_m\geq 1/2$, and cancelling the $D$'s, we have
$$\delta>\exp \left(\frac{-4(d-2)-2(n+3)d^{-1/n}-2d[K:\QQ] (2+\log(n+1))}{d[K:\QQ] d^{1/n} n!^{-1/n}}\right),$$
which contradicts (\ref{Choosing delta}).
\end{proof}

\begin{lemma}\label{existence of D}
In Lemma \ref{auxiliary hypersurface}, we can in fact choose a positive integer $D_0>d-2$ such that for any $D> D_0$, (\ref{positivity}), (\ref{assumption on D, bounded}) and (\ref{another assumption on D}) hold for all $1\leq m\leq n$.
\end{lemma}

\begin{proof}
This can be verified by taking $D\rightarrow \infty$ in (\ref{positivity}), (\ref{assumption on D, bounded}) and  (\ref{another assumption on D}).
\end{proof}

We replace the use of Bertrand's postulate in \cite[Theorem 4.2]{Chen12b} by the use of a theorem on primes between cubes by Dudek \cite{Dud16}, 
in order to obtain polynomial dependence on $d$ in the polylogarithmic factor of the bound.
At the end, we have to bound the size of the covering by complex cells admitting $\delta$-extension. 
This size could be proved to be at most $\poly_{M,n,K}(d,\log B)$ because we were able to choose $\delta$ independent of $B$ in Lemma \ref{auxiliary hypersurface}.
This avoids the issue of overcounting as the factor $B^{(n+1)d^{-1/n}}$ appears when we count points over finite fields using Chen's strategy.

Let $v$ be the least common multiple of $1,2,\dots, [K:\QQ]$.

\begin{thm}\label{main theorem restated}
Let $D$ be an integer such that 
$$D>\max\left\{D_0, 2(M-n)(d-1)+n+2, d^{1+1/n} \frac{nv}{n+1}\cdot 3e^{33.3}\right\}$$ and 
let $B=e^{D/d}$, the set $S_1(X,B)$ of regular $K$-points of $X$ with height $\leq B$ is covered by $N$ hypersurfaces defined over $K$, none of which contain $X$,
where  
$$N\leq \poly_{M,n,K}(d,\log B) \cdot B^{(n+1)d^{-1/n}},$$
and each of which has degree $D=d\log B$.
\end{thm}

\begin{proof}
In the last paragraph of the proof of Theorem 4.2 in \cite{Chen12b}, Chen proved that if $D>2(M-n)(d-1)+n+2$ and if
$$\frac{\log B}{[K:\QQ]}<  \frac{n!}{d(2n+2)^{n+1}} h_{\ol{\mc{L}}} (X)-\frac{3}{2} \log (M+1)-2^n,$$
then 
$S(X,B)$ is contained in a hypersurface of degree $D$ in $\mb{P}^M_K$ which does not contain $X$.
It suffices to prove the case when
$$\frac{\log B}{[K:\QQ]}\geq \frac{n!}{d(2n+2)^{n+1}} h_{\ol{\mc{L}}} (X)-\frac{3}{2} \log (M+1)-2^n.$$
Let $N_0\in (1,\infty)$ be such that
$$\log N_0= d^{-1/n}\frac{n+1}{nv}\log B.$$
Let $r$ be the integral part of the number
$$\frac{(M-n)(d-1)\log B+((M-n)h_{\ol{\mc{L}}}(X)+C_3)[K:\QQ]}{\log N_0}+1.$$
where $C_3$ is the constant defined as follows in Notation 19 of Chen's paper \cite{Chen12b}:
let $\ol{\mc{E}}$ be the trivial Hermitian line bundle of rank $M+1$ over $\Spec \mc{O}_K$, denote by $S^d(\ol{\mc{E}}^\vee)$ be the $d$-th symmetric power of $\ol{\mc{E}}^\vee$, let 
\begin{align*}
C_1:&=(n+2)\widehat{\mu}_{\max}(S^d(\ol{\mc{E}}^\vee))+\frac{1}{2}(n+2)\log \rk (S^d \mc{E})
\\&\quad+\frac{d}{2} \log ((n+2)(M-n))+\frac{d}{2}(n+1)\log (M+1),
\end{align*} 
$$C_2:= \frac{M-n}{2}\log \rk (S^d\mc{E})+\frac{1}{2}\log \rk (\Lambda^{M-n}\mc{E})+\log \sqrt{(M-n)!}+ (M-n)\log d,$$
$$C_3:= (M-n)C_1+C_2.$$
We have 
$$r\leq \frac{A_1\log B+A_2}{\log N_0}+1,$$
where $$A_1=(M-n)(d-1)+\frac{(2n+2)^{n+1}}{n!} (M-n)d$$
and 
$$A_2=[K:\QQ]\left(C_3+\frac{(2n+2)^{n+1}}{n!} d\left(\frac{3}{2}\log (M+1)+2^n\right)\right).$$
By Lemma \ref{existence of D}, the conditions in Lemma \ref{auxiliary hypersurface} hold. 
In particular, $\log B=D/d\geq 1/d.$
We then have
$$r< A_3:=v\frac{A_1+dA_2}{d^{-1/n} (n+1)/n}+1.$$

By our assumption on $D$, we have $N_0^{1/3}>e^{e^{33.3}}$, so by Dudek's Theorem \cite{Dud16},  there exist $r$ distinct prime numbers $p_1,\dots, p_r$ such that $N_0\leq p_i\leq (N_0^{1/3}+r)^3$ for any $i=1,\dots, r$. Choose for each $i$ a maximal ideal $\mf{p}_i$ of $\mc{O}_K$ lying over $p_i$.
By \cite[Lemma 4.1]{Chen12b}, 
$$S_1(X,B)=\bigcup_{i=1}^r S_1(X,B, \mf{p}_i).$$
For any $i$, $N_{\mf{p}_i}$ is a power of $p_i$ whose exponent $f_i$ divides $v$ because $f_i\leq [K:\QQ]$.
Let $a_i:=v/f_i$.
Let $\xi$ be a regular $\mb{F}_{\mf{p}_i}$-point of $\mc{X}_{\mb{F}_{\mf{p}_i}}$.
Let $\mu$ and $\delta$ be as in Lemma \ref{auxiliary hypersurface}.
For each $\sigma:K\hra \CC$, let $\phi_{\alpha(\sigma)}: \mc{C}_{\beta(\sigma)}^{\pmb{\delta}_{\beta(\sigma)}}\ra X(\CC)$ be one of the map constructed in Section \ref{Complex cellular covers}. 
We have
$$\log N_{\mf{p}_i}^{a_i}=\log p_i^{a_if_i}=v \log p_i\geq v\log N_0= d^{-1/n}\frac{n+1}{n}\log B,$$
so by Lemma \ref{auxiliary hypersurface} (the case where $|J|=1$), for any $\eta\in \mc{X}(A_{\mf{p}_i}^{(a_i)})$ whose reduction modulo $\mf{p}_i$ is $\xi$,
there exists a hypersurface of degree $D$ of $\mb{P}^M_K$ not containing $X$ that contains 
$$S(X,B, \eta)\cap \bigcap_{\sigma: K\hra \CC} \phi_{\alpha(\sigma)}( \mc{C}_{\beta(\sigma)}^{\pmb{\delta}_{\beta(\sigma)}}).$$ 
By \cite[Prop. 12.1]{GL02}, the cardinality of $\mc{X}(A_{\mf{p}_i}^{(a_i)})$ is at most 
$$d\sum_{k=0}^n N_{\mf{p}_i}^{a_ik}\leq dn N_{\mf{p}_i}^{a_in}=dnp_i^{f_ia_in}=dnp_i^{vn}\leq dn (N_0^{1/3}+i)^{3vn}.$$
Hence, 
$$S_1(X,B,\mf{p_i})\cap \bigcap_{\sigma: K\hra \CC} \phi_{\alpha(\sigma)}( \mc{C}_{\beta(\sigma)}^{\pmb{\delta}_{\beta(\sigma)}})$$
is covered by at most $dnp_i^{vn}\leq dn (N_0^{1/3}+i)^{3vn}$ hypersurfaces of degree $D$ of $\mb{P}^M_K$ not containing $X$.
Therefore, 
$$S_1(X,B)\cap \bigcap_{\sigma: K\hra \CC} \phi_{\alpha(\sigma)}( \mc{C}_{\beta(\sigma)}^{\pmb{\delta}_{\beta(\sigma)}})$$ 
is covered by at most 
$$\sum_{i=1}^r dn (N_0^{1/3}+i)^{3vn}\leq rdn(N_0^{1/3}+r)^{3vn}\leq \poly_{n,K}(r,d) B^{(n+1)d^{-1/n}}$$
hypersurfaces of degree $D$ of $\mb{P}^M_K$ not containing $X$.
By (\ref{size of cellular cover}), $S_1(X,B)$ is covered by $N$ hypersurfaces of degree $D$ of $\mb{P}^M_K$ not containing $X$, where $N$ is at most
$$ \poly_{n,K}(r,d) B^{(n+1)d^{-1/n}}\cdot (M+1)^{[K:\QQ]}\cdot \poly_{M,K}(d)\cdot \poly_{M,K} (\mu \log (1/\delta))\cdot \delta^{-2n[K:\QQ]}.$$
From \cite[Lemma 4.3]{Bos01}, we know $\widehat{\mu}_{\max}(S^d(\ol{\mc{E}}^\vee))\leq cd$, where $c$ is a constant depending only on $M$ and $K$.
We have $\log \rk (S^d\mc{E})\leq \rk(S^d\mc{E})\leq \poly_M(d)$.
We thus have $C_3\leq \poly_{M,n,K}(d)$, 
so $r<A_3\leq  \poly_{M,n,K}(d)$.
Recall that $\mu=r_1(D)$ and $D=d\log B$. 
Thus by (\ref{Chardin's bound}), $\mu\leq \poly_{n}(d, \log B)$.
We have 
\begin{align*}
\delta^{-1}
&\leq \max\left\{4, \exp \left(\frac{4(d-2)/d+2(n+3)d^{-1-1/n}+2[K:\QQ] (2+\log(n+1))}{[K:\QQ] d^{1/n} n!^{-1/n}}\right)\right\}
\\& \leq \max\left\{4,\exp \left(\frac{4+2(n+3)+2[K:\QQ] (2+\log(n+1))}{[K:\QQ] n!^{-1/n}}\right)\right\}.
\end{align*}
Therefore, $N$ is at most
$\poly_{M,n,K}(d, \log B) \cdot B^{(n+1)d^{-1/n}}.$
\end{proof}

Theorem \ref{Arakelov Bombieri-Pila} follows from Theorem \ref{main theorem restated} by taking 
$$B_0:=\exp \left(\frac{1}{d}\max\left\{D_0, 2(M-n)(d-1)+n+2, d^{1+1/n} \frac{nv}{n+1}\cdot 3e^{33.3}\right\}\right).$$


\begin{bibdiv}
\begin{biblist}

\bib{BCK24}{article} {author={G. Binyamini}, author={R. Cluckers}, author={F. Kato}, title={Sharp bounds for the number of rational points on algebraic curves and dimension growth, over all global fields},   pages={arXiv:2401.03982v2}}

\bib{BN19}{article} {author={G. Binyamini}, author={D. Novikov}, title={Complex cellular structures},  journal={Ann. of Math.}, volume={190(1)}, date={2019},  pages={145--248}}

\bib{BP89}{article}{author={E. Bombieri}, author={J. Pila}, title={The number of integral points on arcs and ovals}, journal={Duke Math. J.}, volume={59(2)}, date={1989},  pages={337--357}}

\bib{Bos96}{article} {author={J.-B. Bost}, title={P\'{e}riodes et isog\'{e}nies des vari\'{e}t\'{e}s ab\'{e}liennes sur les corps de nombres [d'apr\`{e}s D. Masser et G. W\"{o}stholz]},  journal={S\'{e}minaire Bourbaki,  Vol. 1994/1995, Ast\'{e}risque}, volume={237}, date={1996},  pages={Exp. No. 795, 115--161}}

\bib{Bos01}{article} {author={J.-B. Bost }, title={Algebraic leaves of algebraic foliations over number fields},  journal={Publ. Math. de l' IH\'{E}S}, volume={93}, date={2001},  pages={161--221}}

\bib{Bos06}{article} {author={J.-B. Bost}, title={Evaluation maps, slopes, and algebraicity criteria},  journal={Proc. Internat. Congr. Math. (Madrid 2006), Euro. Math. Soc.}, date={2007},  pages={537--562}}

\bib{Bro04}{article}{author={N. Broberg}, title={A note on a paper by R. Heath-Brown: ``The density of rational points on curves and surfaces''}, journal={J. reine angew. Math.}, volume={571}, date={2004},  pages={159--178}}

\bib{Cha89}{article} {author={M. Chardin}, title={Une majoration de la fonction de Hilbert et ses cons\'{e}quences pour l'interpolation alg\'{e}brique},  journal={Bulletin de la S. M. F.}, volume={117(3)}, date={1989},  pages={305--318}}

\bib{Chen12a}{article} {author={H. Chen}, title={Explicit uniform estimation of rational points I. Estimation of heights},  journal={Journal f\"{u}r die reine und angewandte Mathematik (Crelles Journal)}, volume={668}, date={2012},  pages={59--88}}

\bib{Chen12b}{article} {author={H. Chen }, title={Explicit uniform estimation of rational points II. Hypersurface coverings},  journal={Journal f\"{u}r die reine und angewandte Mathematik (Crelles Journal)}, volume={668}, date={2012},  pages={89--108}}

\bib{Dud16}{article} {author={A. W. Dudek}, title={An explicit result for primes between cubes},  journal={Funct. Approx. Comment. Math.}, volume={55(2)}, date={2016},  pages={177--197}}

\bib{GL02}{article} {author={S. R. Ghorpade}, author={G. Lachaud}, title={\'{E}tale cohomology, Lefschetz Theorems and Number of Points of Singular Varieties over Finite Fields},  journal={Mosc. Math. J.}, volume={2(3)}, date={2002},  pages={589--631; \emph{corrigenda and addenda}, Mosc. Math. J. \textbf{9(2)} (2009), 431--438.}}

\bib{Hea02}{article}{author={D. R. Heath-Brown}, title={The Density of Rational Points on Curves and Surfaces}, journal={Ann. of Math.}, volume={155(2)}, date={2002},  pages={553--598}}

\bib{Pil95}{article} {author={J. Pila}, title={Density of integral and rational points on varieties},  journal={Ast\'{e}risque}, volume={228}, date={1995},  pages={183--187}}

\bib{Sal07}{article} {author={P. Salberger}, title={On the Density of Rational and Integral Points on Algebraic Varieties},  journal={J. Reine Angew. Math.}, volume={606}, date={2007},  pages={123--147}}

\bib{Sar88}{article} {author={P. Sarnak}, title={Torsion points on varieties and homology of Abelian covers},  journal={Manuscript},  date={1988}}

\bib{Som97}{article} {author={M. Sombra}, title={Bounds for the Hilbert function of polynomial ideals and for the degrees in the Nullstellensatz},  journal={Journal of Pure and Applied Algebra}, volume={117/118}, date={1997},  pages={565--599}}

\bib{Wal15}{article} {author={M. N. Walsh}, title={Bounded Rational Points on Curves},  journal={Int. Math. Res. Not.}, volume={(14)}, date={2015},  pages={5644--5658}}

\end{biblist}
\end{bibdiv}
\end{document}